\documentclass{daj}
\usepackage{amsmath}
        \usepackage{amsmath}
        \usepackage{amsfonts}
        \usepackage{amssymb}
        \usepackage{amsthm}
        \usepackage{mathrsfs}
        \usepackage{lscape}
        \usepackage{latexsym,enumerate,color,mdwlist}
        \usepackage{tikz}
        \usepackage{a4wide}
        \usepackage{hyperref}
        \usepackage{soul}
        \usepackage{mathtools}
        
        \usepackage[shortlabels]{enumitem}
        \setlist[enumerate]{leftmargin=25pt}
        \setlist[itemize]{leftmargin=25pt}
        
        \theoremstyle{plain} \newtheorem{Thme}{Theorem}[section]
        \newtheorem{prop}[Thme]{Proposition}
        \newtheorem{lem}[Thme]{Lemma}

        \theoremstyle{definition}

        \newtheorem*{theorem*}{Theorem}
        \theoremstyle{remark}

        \newcommand{\comment}[1]{}

        \def\N{\mathbb{N}}
        
        \def\Z{\mathbb{Z}}
        
        \def\R{\mathbb{R}}
        \def\La{\Lambda}
        \def\la{\lambda}

\dajAUTHORdetails{%
  title = {Spectral sets in $\mathbb{Z}_{p^2qr}$}, 
  author = {G\'abor Somlai},
  plaintextauthor = {Gabor Somlai},
  plaintexttitle = {Spectral sets in Z{p2qr}}, 
  runningtitle = {Spectral sets in }, 
  runningauthor = {G\'abor Somlai},
  copyrightauthor = {G. Somlai},
  keywords = {spectral set, tile},
}   

\dajEDITORdetails{%
   year={2023},
   number={5},
   received={27 September 2021},   
   revised={23 August 2022},    
   published={25 May 2023},  
   doi={10.19086/da.77359},       
}   

\begin{document}

\begin{frontmatter}[classification=text]

\title{Spectral sets in $\mathbb{Z}_{p^2qr}$} 

\author[pgom]{G\'abor Somlai \thanks{Research was supported by the J\'anos Bolyai Research Fellowship of the Hungarian Academy
        of Sciences and the New National Excellence Program under the grant number UNKP-20-5-ELTE-231, received funding from the European Research Council (ERC) under the European
        Union’s Horizon 2020 research and innovation programme (grant agreement No. 741420)}}

\begin{abstract}
        We prove that every spectral set in $\Z_{p^2qr}$ tiles, where $p$, $q$ and $r$ are primes. Combining this with a recent result of Malikiosis \cite{Romanos2020} we obtain that Fuglede's conjecture holds for $\Z_{p^2qr}$.
        \end{abstract}
\end{frontmatter}


        \section{Introduction}
        A Lebesgue measurable set $\Omega \subset \R^n$ is called a {\it tile} if there is set $T$ in $\R^n$, called the {\it tiling complement} of $\Omega$ such that almost every element of $\R^n$ can uniquely written as $\omega+t$, where $\omega \in \Omega $ and $t \in T$.
        We say that $\Omega$ is {\it spectral} if $L^2(\Omega)$ admits an orthogonal basis consisting of exponential functions of the form $\{ e^{i (\lambda  x)} \mid \lambda \in \Lambda, x \in \Omega \}$, where $\La \subset \R^n$. In this case $\La$ is called a {\it spectrum } for $\Omega$.

        Fuglede conjectured \cite{F1974} that a bounded domain $S \subset \mathbb{R}^d$ tiles the $d$-dimensional Euclidean space if and only if the set of $L^2(S)$ functions admits an orthogonal basis of exponential functions. The conjecture might have been motivated by Fuglede's result \cite{F1974} that it is true when the tiling complement or the spectrum is a lattice in
        $\R^n$. The conjecture was disproved in \cite{T2004}. Tao considered a discrete version of the original conjecture and constructed spectral sets in $\Z_2^{11}$ and $\Z_3^5$, which are not tiles. The latter example was lifted to $\R^5$ to refute the spectral-tile direction of Fuglede's conjecture. 
        Matolcsi \cite{M2004} proved that the spectral-tile direction of the conjecture
        fails in $\R^4$. Kolountzakis and Matolcsi
        \cite{KMkishalmaz,KM2006a} and Farkas, Matolcsi and M\'ora \cite{FMM2006}
        provided counterexamples in $\R^3$ for both 
        directions of the conjecture. 
        
        It is important to note that every (finite) tile of the integers is periodic \cite{cm}. Moreover, the tile-spectral direction of the conjecture is true for $\R$ if and only it holds for $\Z$, which is further equivalent for the conjecture to hold for every finite cyclic group, see \cite{DL2014}. If the spectral-tile direction of Fuglede's conjecture holds for $\R$, then it holds for $\Z$, which again implies that it holds for every (finite) cyclic group. On the other hand, the converse of this statements is not necessarily true. 
        
        The recent investigations of Fuglede's conjecture for finite cyclic groups started with a result of \L aba \cite{laba} proving that both directions of Fuglede's conjecture holds for $\mathbb{Z}_{p^k}$ for every positive integer $k$. It was also proved in \cite{laba} that tile-spectral direction holds for cyclic groups whose order is divisible by (at most) two different primes.
        Kolountzakis and Malikiosis \cite{Kol} proved that the conjecture holds for $\Z_{p^k q}$, where $p^k$ is an arbitrary power of the prime $p$ and $q$ is a prime. As a strengthening it was proved in \cite{negyen} that Fuglede's conjecture holds for $\Z_{p^nq^2}$, where $p$ and $q$ are prime. A recent result of Malikiosis shows that a for cyclic group of order $p^mq^n$, Fuglede's conjecture holds if $ \min \{m,n \} \le 6$ or $p^{m-2}<q^4$.
        Another important result was proved by Shi \cite{shi} who showed that Fuglede's conjecture is true for $\Z_k$ if $k$ is the product of $3$ (distinct) primes. 
        
        The so-called Coven-Meyerowitz conjecture states that if $n$ is square-free, then every tile of $\Z_n$ is a complete set of residues$\pmod{k}$, where $k$ is a divisor of $n$. This was originally settled (in two short comments) by \L aba and Meyerowitz on Tao's blog and a self-contained proof of this fact was provided by Shi \cite{shi}. 
        
        Coven and Meyerowitz \cite{cm} proved that a subset $A$ of $\Z_n$ tiles if  two properties called (T1) and (T2), defined in the next section, are satisfied. 
        The converse also holds if $n$ has at most two different prime divisors or if $n=p^m q_1 \ldots q_k$ where $p, q_1, \ldots ,q_k$ are different primes \cite{Romanos2020}. Furthermore it was also proved in \cite{cm} that for every $n \in \N$ the tiles of $\Z_n$ satisfy (T1).
        
        The main result of the paper is the following. 
        \begin{Thme}
        Every spectral set in $\Z_{p^2qr}$ is a tile. 
        \end{Thme}
        Combining this theorem with the result of Malikiosis mentioned above we obtain that Fuglede's conjecture holds for $\Z_{p^qr}$.
        \section{Fuglede's conjecture for cyclic groups}
        Fuglede's conjecture is still open in the $1$ and $2$ dimensions. We will focus on the one dimensional case which is heavily connected to the discrete version of the conjecture for cyclic groups. 
        
        Let $S $ be a subset of $\Z_n$. We say that $S$ is a tile if and only if there is a $T \subset \Z_n$ such that $S+T=\Z_n$ and $|S||T|=n$. 
        We say that $S$ is spectral if and only if the vector space of complex functions on $S$ is spanned by pairwise orthogonal functions, which are the restrictions of some irreducible representations of $\Z_n$. Note that cyclic groups are abelian thus the irreducible representations considered here are all one dimensional so these are just characters. The irreducible representations of $\Z_n$ are of the following form: 
        \[ \chi_k(x)=e^{\frac{2 \pi i k}{n}x},\]
        so these are parametrized by the elements of $\Z_n$. It is easy to verify that $\chi_k$ and $\chi_l$ are orthogonal if and only if $\chi_{k-l}$ is orthogonal to the trivial representation, which can also be written as 
        \[ \sum_{s\in S} \chi_{k-l}(s)=0.
        \]
        One can assign a polynomial $m_S$ to $S$ by $\sum_{s \in S}x^s$, which is called the {\it mask polynomial} of $S$. It is easy to see that $\sum_{s\in S} \chi_{k}(s)=0$ if and only if $\xi_{\frac{n}{gcd(k,n)}}$ is a root of $m_S$, where $\xi_d$ is a primitive $d$'th root of unity. This can also be written as $\Phi_d \mid m_S$, where $\Phi_d$ is the $d$'th cyclotomic polynomial. 
        Note that mask polynomials can also be defined for any element of the group ring $\Z[\Z_n]$.
        
        Now using the character table of $\Z_n$ we have that $(S,\La)$ is a spectral pair if and only if the submatrix of the character table whose rows are indexed by the elements of $\La$ and columns by those of $S$ is a complex Hadamard matrix. In fact, the adjoint of a complex Hadamard matrix is also a complex Hadamard matrix so if $(S,\La)$ is a spectral pair, then $(\La,S)$ is a spectral pair too. 
        
        Now we introduce the properties that are needed to formulate the Coven-Meyerowitz conjecture (first appeared in a paper of \L aba and Konyagin \cite{labakonyagin}) for any cyclic group. Let $H_S$ be the
        set of prime powers $p^k$ dividing $N$ such that $\Phi_{p^k}(x) \mid
        m_S(x)$.
        \begin{enumerate}[{\bf(T1)}]
             \item\label{t1} $m_S(1)=\prod_{d \in H_S} \Phi_d(1)$,
         \item\label{t2} for pairwise relatively prime elements $q_i$ of $H_S$, we have $\Phi_{\prod q_i} \mid m_S(x)$.
        \end{enumerate}
        We remind that if (T1) and (T2) hold for some $S \subset \Z_n$, then $S$ is a tile and if $S$ is a tile, then (T1) holds, see \cite{cm}. Further we mention that \L aba \cite{laba} proved that a set having (T1) and (T2) properties also is a spectral set.
        \section{Preliminary lemmas}
        For the sake of simplicity let $n=p^2qr$.
        First, we collect the results obtained in \cite{negyen} that apply in our case. Note that in our case ($n=p^2qr$) for every proper subgroup or quotient group of $\Z_n$ we have that spectral sets coincide with tiles. 
        The results of Section 4 in \cite{negyen} were summarised in a statement called Reduction 1 but it is important to note that $n$ has more than $2$ different prime divisors so we only have the following. 
        \begin{prop}\label{propreduction}
        Let us assume that $(S,\La)$ is a spectral pair for an abelian group whose subgroups and factor groups satisfy the spectral-tile direction of Fuglede's conjecture. Without loss of generality we may assume $0 \in S$, $0 \in \La$. Further $S$ is a tile if one of the following holds. 
        \begin{enumerate}
            \item\label{itemreductiona} $S$ or $\La$ does not generate $\Z_n$, 
            \item\label{itemreductionb} $S$ can be written as the union of $\Z_u$-cosets, where $u$ is a prime dividing $n$.
        \end{enumerate}
        \end{prop}
        \begin{lem}\label{lempqnemoszt}
        Let $0 \in T \subset \Z_N$ a generating set and assume $x$ and $y$ are different prime divisors of $N$. Then there are elements $t_1 \ne t_2$ of $T$ such that $x \nmid t_1-t_2$ and $y \nmid t_1-t_2$ for any pair of prime divisors of $N$.
        \end{lem}
        \begin{proof}
        $T$ is not contained in any proper coset of $\Z_N$ so it contains an element $t_1$ not divisible by $x$ and $t_2$ not divisible by $y$. If $y \nmid t_1$, then $t_1-0 \in (T -T)$ is not divisible by either $x$ or $y$, when we are done so we may assume $y \mid t_1$. 
        Similar argument shows that we may assume $x \mid t_2$. Then $x \nmid t_1-t_2$ and $y \nmid t_1-t_2$, as required. 
        \end{proof}
        Another important tool is the following lemma. This is the same as Proposition 3.4 in \cite{negyen} formulated in a different language. Let $m$ be a square-free integer, where $m$ is the product of $d$ primes. Then $\Z_m \cong \bigoplus_{i=1}^d\Z_{p_i}$ a direct sum of $d$ cyclic groups of different  prime order so the elements of $\Z_m$ are encoded by $d$-tuples. This allows us to introduce Hamming distance on $\Z_m$. Further we say that $P$ is a cuboid in $\Z_m$ if it can be written as $\prod H_i$, where $H_i \subset \Z_{p_i}$ with $|H_i|=2$. 
        \begin{prop}\label{prophypercube}
        Let $w$ be an element of the group ring $\Z[\Z_m]$ with nonnegative coefficients, where 
        $m=\prod_{i=1}^d p_i$, a product of $d$ different primes. Assume $\Phi_m \mid m_w$. Let $P$ be a $d$-dimensional
        cuboid and $p$ a vertex of $P$. Then
        \begin{equation}\label{eqparallelograme}
        \sum_{c \in P} (-1)^{d_H(p,c)}w(c)=0.
        \end{equation}
        \end{prop}
        We will refer to the previous proposition or more precisely to equation \eqref{eqparallelograme} as the $d$-dimensional cube-rule. Note that Proposition \ref{prophypercube} is a corollary of Corollary 3.4 in \cite{negyen}, that is formulated below. 
        \begin{lem}\label{lem0603}
        Let $w$ be an element of the group ring $\Z[\Z_m]$ with nonnegative coefficients. Assume $\Phi_m \mid m_{\omega}$.
        \begin{enumerate}
        \item\label{item:lem0603a} Then $w$ can be written as the weighted sum of $\Z_{p_i}$-cosets with rational coefficients. 
        \item If $m=pq$, where $p$ and $q$ are different primes, then $w$ can be written as the $\Z_p$-cosets and $\Z_q$-cosets with nonnegative integer coefficients.
        \end{enumerate}
        \end{lem}
        In fact, the coefficients in Lemma \ref{lem0603} \ref{item:lem0603a} can be chosen to be integers.

        It is also important to know what happens if $\omega$ is a multiset on $\Z_m$, with $\Phi_m \mid m_{\omega}$, where $m$ is not square free, which is also described in \cite{negyen}. Let $m'$ be the radical of $m$. Then $\omega$ is the weighted sum of $\Z_{p_i}$-cosets with integer coefficients again, implying that the cube-rule holds for the restriction of $\omega$ to each $\Z_{m'}$-coset.   
        \begin{lem}\label{lemcuberulecorr}
        Let $T \subset \Z_{pqr}$. Assume $T$ satisfies the $3$-dimensional cube-rule and $T \cap ((t+\Z_q) \cup (t+\Z_r))=\{t\}$ for all $t\in T$, where $\Z_q$ and $\Z_r$ denotes the unique subgroups of $\Z_{pqr}$ of order $q$ and $r$, respectively. Then $T$ is a union of $\Z_p$-cosets.
        \end{lem}
        Note that the statement holds verbatim for any permutation of the primes $p,q,r$.
        \begin{proof}
        Assume $\Phi_{pqr} \mid m_T$, then $T$ satisfies the $3$ dimensional cube-rule. Suppose $t \in T$. By our assumption $T \cap (t+\Z_q)=\{t\}$ and $T \cap (t+\Z_r)=\{t\}$.   
        
        By the way of contradiction, assume $T$ is not a union of $\Z_{p}$-cosets so there is $t \in T$ with $t_p \in (t+\Z_p)\setminus T$. Then for every cuboid containing $t$ and $t_p$ as its vertices, the neighbours (consider the cube as the natural graph on $8$ vertices) of $t$ are not in $T$. Now using the cube-rule we obtain that the vertex of the cuboid, which is of Hamming distance $3$ from $t$ is contained in $T$.
        Thus for every $x \in \Z_{pqr}$ with $p \mid x-t_p$ and $d_H(x,t)=3$ we have $x \in T$. Then there are elements of $T$ whose difference is divisible by $pr$ if $q>2$  and the same holds with  $pq$ if $r>2$. This contradicts the assumption that $T \cap ((x+\Z_q) \cup (x+\Z_r))=\{x\}$.
        \end{proof}
        Now we prove a Lemma that will be used in the proof of our main result. Since $\Z_n$ is a cyclic group, for every $m \mid n$ there is a unique subgroup of $\Z_n$ of order $m$. Thus if $f$ is a function on $\Z_n$, then there is a well defined way to define its projection to $\Z_{\frac{n}{m}}$ by summing the values of $f$ on each $\Z_m$-coset. This can also be applied to sets by identifying them with their characteristic function. 
        \begin{lem}\label{lemprojekcio}
        Let $T \subset \Z_N$, where $N$ is a positive integer and let $m $ and $r$ be divisors of $N$, where $r$ is a prime with $(m,r)=1$. Assume further that for every $1 < d \mid m$ we have $\Phi_d \mid m_T$ or $\Phi_{dr} \mid m_T$. Let $T_m$ denote the projection (points counted with multiplicity) of $T$ to $\Z_m$. Then $T_m=c\Z_m+rD$, where $c \ge 0$ an integer and $D$ is a multiset on $\Z_m$. 
        \end{lem}
        \begin{proof}
        Since $(d,r)=1$ we obtain $\Phi_{dr}(x)=\frac{\Phi_d(x^r)}{\Phi_d(x)}$. This equation can be written as $\Phi_{dr}(x)=\Phi_d(x)^{r-1}$ in the polynomial ring $\Z_r[x]$. Now $\Phi_{dr}(x) \mid m_T(x)$ implies $\Phi_d(x) \mid m_T(x)$ in $\Z_r[x]$. These cyclotomic polynomials are pairwise
        relatively prime in $\Z_r[x]$ as well so we obtain $\prod_{1 < d \mid m} \Phi_d \mid m_T$ in $\Z_r[x]$. This implies first that 
        $\prod_{1 <d \mid m} \Phi_d \mid m_T$ in $\mathbb{Z}_r[x]$, so that $$ m_T(x) \equiv c \left( 1+x+ \cdots +x^{m-1} \right) \pmod{x^m-1}$$ in $\Z_r[x]$, for some $0 \le c <r$. If we also have $\Phi_1(x) \mid m_T$ in $\mathbb{Z}_r[x]$, then we have $c=0$.
        \end{proof}
        \section{Proof of the main result} 
        Let $(S, \La)$ be a spectral pair. We will distinguish certain different cases by the cardinality of $S$, which is equal to that of $\La$. For a subset $A \subseteq \Z_n$ we write $k \mid \mid |A|$ if $gcd(|A|,n)=k$.
        
        Before we start proving our main result we introduce a notation. 
        For every $k \mid n$ there is a unique subgroup of order $k$ of $\mathbb{Z}_n$, which we may also denote by $\mathbb{Z}_k$. For a subset $A$ of $\mathbb{Z}_n$ one can define a function from the cosets of $\mathbb{Z}_k$ to $\mathbb{N}$. The image of a coset is the number of elements of $A$ contained in the coset. This function can also be considered as a multiset i.e. elements of $\Z[\Z_{\frac{n}{k}}]$ with nonnegative coefficients, and it will be denoted by $A_k$.
        Sometimes we say that $A_k$ is the natural projection of $A$ to $\mathbb{Z}_{\frac{n}{k}}$.
        \subsection{$|S|$ has 3 prime divisors} 
        In this case we assume that $|S|$ is divisible by three of the primes $p,q,r$ counted with multiplicity. 
        Thus the cases handled here are $p^2q \mid \mid |S|$, $p^2r \mid \mid |S|$ and $pqr \mid \mid |S|$. 
        
        Assume  $p^2q\mid\mid |S|$.  Project $S$ to $\Z_{p^2q}$. If two elements of $S$ project to the same element of $\Z_{p^2q}$, then we have a pair of elements of $s_1, s_2 \in S$ with $p^2q \mid \mid s_1-s_2$. 
        Note that this is always the case when $|S|>p^2q$. Thus we have $\Phi_r \mid m_{\La}$, which implies $r \mid |\La|=|S|$, a contradiction.
        Thus $|S|=p^2q$ and $S$ is a complete set of residues $\pmod{p^2q}$ so is a tile. 
        
        The same argument works if $p^2r\mid\mid |S|$.
        
        Let us assume now that $pqr \mid \mid |S|$. If a $\Z_{p^2}$-coset contains at least $p+1$ elements of $S$, then $S$ contains a pair of elements contained in this $\Z_p$-coset and a pair of elements contained in the same $\Z_{p^2}$-coset but in different $\Z_p$-cosets. These imply $\Phi_p \mid m_{\La} $ and $\Phi_{p^2} \mid m_{\La} $ so $p^2 \mid |\La|=|S|$, a contradiction.  
        
        It remains to investigate the case when each $\Z_{p^2}$-coset contains exactly $p$ elements of $S$, which gives $|S|=pqr$. Moreover by excluding $\Phi_p \Phi_{p^2} \mid m_{\La}$ we obtain that 
        the intersection of $S$ with each $\Z_{p^2}$-coset is either a $\Z_p$-coset or it is a complete coset representative of $\Z_p$-cosets. In both of these cases $S$ is a tile. 
        
        Note that similar argument will also be used later in this paper if $|S|>p^2 \min \{q,r\}$ or $|S|>pqr$. 
        \subsection{$|S|$ has two prime divisors} Now we handle the cases when $p^2 \mid \mid |S|$, $pq \mid\mid |S|$, $pr \mid\mid |S|$ and $qr \mid\mid |S|$. 
        
        {\bf Case 1.} 
        We first handle the case $p^2 \mid \mid |S|$. \\
        Assume first that  $\Phi_n \mid m_S$. Then we may apply the cube-rule on every $\Z_{pqr}$-coset. By Lemma \ref{lemcuberulecorr}
        we obtain that $S$ is the union of $\Z_p$-cosets, which case is handled in Proposition \ref{propreduction}  \ref{itemreductionb}. Thus we may assume $\Phi_n \nmid m_S$.

        If every $\Z_{qr}$-coset contains exactly one element of $S$, then $S$ is a tile. Thus we may assume there are $s_1 \ne s_2 \in S$ with $p^2 \mid s_1-s_2$. If $p^2q \mid\mid s_1-s_2$ or $p^2r \mid\mid s_1-s_2$, then $r \mid |S|$ or $q \mid |S|$, respectively. Both of these cases contradict 
         $p^2 \mid \mid |S|$.
         Thus we have $p^2 \mid \mid s_1-s_2$ so $\Phi_{qr} \mid m_{\Lambda}$.

         It follows from $\Phi_{qr} \mid m_{\La}$ that $\La_{qr}$ is the weighted sum of $\Z_q$-cosets and $\Z_r$-cosets with nonnegative weights by Lemma \ref{lem0603}. Both type of cosets appear since otherwise 
        we would have $q  \mid |S|$ or $r  \mid |S|$, a contradiction. 
        
        Let $\Gamma$ be a graph whose vertices are the elements of $\La$ and two vertices are adjacent if and only if their difference is not divisible by either $q$ or $r$. 
        Note that we have already assumed $\Phi_n \nmid m_S$ which implies $p \mid \lambda-\lambda '$ or $q \mid
        \lambda-\lambda '$ or $r \mid \lambda-\lambda '$ for every
        $\lambda, \lambda ' \in \Lambda$.
        Thus if $\Gamma$ is connected, then we have $\La$ is contained in a $\Z_{pqr}$-coset, which is excluded by Proposition \ref{propreduction}.

        Without loss of generality we may assume $r>q$ so $r\ge 3$. 
        Then there are $\la, \la' \in  \La$ with  $q \mid \la-\la'$ but  $r \nmid \la-\la'$ and there is $\la'' \in \La$, whose $q$ and $r$-coordinates differ from those of $\la$ and $\la'$. By this we mean $q \nmid \la-\la''$, $q \nmid \la'-\la''$, $r \nmid \la-\la''$ and $r \nmid \la'-\la''$. 
        Since $\Phi_n \nmid m_S$ we have $p  \mid \la-\la''$ and $p  \mid \la'-\la''$ so we have $pq  \mid \la-\la'$ and $r \nmid \la-\la'$. 
        Thus we either have $\Phi_r \mid m_S$, which is excluded since $r \nmid |S|$, or we have $\Phi_{pr} \mid m_S$. 

        Assume $\Gamma$ is disconnected. Let $\tilde{\Lambda}_{qr}$ denote the underlying set of the multiset $\Lambda_{qr}$. We have that $\La_{qr}$ is the sum of $\Z_q$ and $\Z_r$-cosets and we have seen that  both types appear. Thus $\tilde{\La}_{qr}$ is the union of a $\Z_q$-coset $Q$ and a $\Z_r$-coset $R$, otherwise $\Gamma$ is connected. 
        In this case, $\Gamma$ has a large connected component consisting of those elements of $\La$ which do not project to the
        intersection of $Q$ and $R$ denoted by $x \in \Z_{qr}$. Thus these points are contained in a single $\Z_{pqr}$-coset. Note that the number of elements of $\La$ projecting to $x$ is less than $\frac{|\La|}{2}$.
        
        We claim that $\Phi_n \nmid m_{\La}$. By the way of contradiction let us assume $\Phi_n \mid m_{\La}$. Using the same argument when we excluded $\Phi_n \mid m_S$ we have that $\La$ is the union of $\Z_p$-cosets. But then $(\La,S)$ is a spectral pair and by Proposition \ref{propreduction} we have that $\La$ is a tile whence $|\La|=p^2=|S|$. 
         Theorem B1 in \cite{cm} shows that (T1) holds for $\La$. Thus it follows that $\Phi_p \Phi_{p^2} \mid m_{\La}$. Since $\Phi_p \Phi_{p^2}=1+x+ \ldots +x^{p^2-1}$ and $|\La|=p^2$ we have $\La$ is a complete set of coset representatives of $\Z_{qr}$ in $\Z_n$. This contradicts the fact that all but the elements of $\La$ projecting to $x \in \Z_{qr}$ give the same remainder$\pmod{p}$. 
        
        We have that there is $\mu$ in $\La$ projecting to $x \in \Z_{qr}$ with $p \nmid \mu-\la$ and $p \nmid \mu-\la''$, where $\la$ projects to $R \setminus\{x\}$ and $\la'$ to $Q \setminus\{x\}$, otherwise $\La$ is contained in a $\Z_{pqr}$-coset. 
        Thus we obtain $\Phi_{p^2q} \mid m_S$ and $\Phi_{p^2r} \mid m_S$ since $\Phi_n \nmid m_S$.
        
        Now we exclude the case $p=2$. If $p=2$, then by $r>q$ we have $r \ge 5$. Thus the description of $\tilde{\La}_{qr}$ shows that there are at least $4$ elements of $\La$ projecting to $R$ in the same connected component of $\Gamma$. The difference of any two of these is divisible by $p$ so there is a pair whose difference is divisible by $p^2$ as well. Since their projections to $\Z_{qr}$ lie in $R$ we obtain $\Phi_r \mid m_S$, a contradiction. 
        
        We claim that $p>r$. Otherwise, since $p=2$ is excluded we have $r-p \ge 2$. Hence there are more than $p$ elements of $\La$, projecting to mutually different points of $R \setminus \{x\}$. Their difference is divisible by $p$ since they are in the same connected component of $\Gamma$ but then there would be a pair of elements of $\Lambda$ whose difference is divisible by $p^2$ as well, implying $\Phi_r \mid m_S$ and $r \mid |S| $, a contradiction. 
        Now $p>q$ follows from our assumption that $r > q$.
        
        
        A simple calculation shows that the number of elements $m$ of $\La$ projecting to $x$ exceeds $p$. This follows from $p^2 \le kq+lr$ and $m=k+l$ since $p>q,r$, where $k$ and $l$ are defined by $\La_{qr}=kQ+lR$. Thus there are elements $\la_3,\la_4,\la_5$ of $\La$ with $qr \mid \mid \la_3-\la_4$ and $pqr \mid \mid \la_3-\la_5$ thus implying $\Phi_p \Phi_{p^2} \mid m_S$.

        We remind that $\Phi_{pr} \mid m_S$.
        The fact that $\Phi_p \Phi_{p^2} \mid m_S$ implies that every $\Z_{qr}$-coset contains the same amount of elements of $S$. Denote this number by $a$. If $a=1$, then $S$ is a tile so we assume $a \ge 2$. 
        If $a=q$, then $|S|=p^2q$ which case has been handled before. If $a>q$, then $|S|>p^2q$ and by a simple pigeonhole argument  we obtain $r \mid |S|$.
        
        Now project $S$ to $\Z_{p^2r}$. $S_{p^2r}$ is a set, otherwise $q \mid |S|$.
        Then each $\Z_r$-coset in $\Z_{p^2r}$ contains $2 \le a < q$  elements of $S_{p^2r}$. 
        Using $\Phi_{p^2r} \mid m_S$ 
        we have that the intersection of $S_{p^2r}$ with each $\Z_{pr}$-coset is the union of $\Z_p$-cosets or $\Z_r$-cosets. Since $a<q<r$ 
        it is the union of $\Z_p$-cosets only. 
        
        Then we build up a graph $\Gamma'$ whose vertices are the elements of $\La$ and two vertices are adjacent if and only if their difference is not divisible by either $p$ or $r$. It follows from the previous observations using $p \ge 3$ that $\Gamma'$ is connected. 
        Since $\Phi_n \nmid m_\La$ we have that the $q$-coordinates of these elements of $\La$ are the same so $\La$ is contained in a proper coset of $\Z_n$. Thus by Proposition \ref{propreduction} we have that $S$ is a tile. 
        
        {\bf Case 2.} Assume $pq \mid\mid |S|$.
        We  may exclude the case, when  $S$ is complete set of residues$\pmod{pq}$, which is the same as $S$ contains exactly one element from each $\Z_{pr}$-coset, since $S$ is a tile in this case. Then there are elements of $S$ projecting to the same element of $\Z_{pq}$. We either have $\Phi_{pr} \mid m_{\La}$ or $\Phi_p \mid m_{\La}$ or $\Phi_{r} \mid m_{\La}$.  The latter case is impossible since $r \nmid |S|$.
        
        Now we project $S$ to $\Z_{p^2q}$. The projection $S_{p^2q}$ is a set, otherwise we have $\Phi_r  \mid m_{\La}$ and thus $r \mid |S|$, a contradiction. If there is a $\Z_{p^2}$-coset of $\Z_{p^2q}$ containing more than $p$ elements of $S_{p^2q}$, then there are two among them, whose difference is not divisible by $p$, implying $\Phi_{p^2} \mid m_{\La}$ or $\Phi_{p^2r} \mid m_{\La}$.
        If every $\Z_{p^2}$-coset of $\Z_{p^2q}$ contains exactly $p$ elements of $S_{p^2q}$ and for each $\Z_{p^2}$-coset all of these elements are contained in the same $\Z_p$-coset, then we have that $S$ is a tile. Thus we may assume $\Phi_{p^2} \mid m_{\La}$ or $\Phi_{p^2r} \mid m_{\La}$.
        
        Applying Lemma \ref{lemprojekcio} using the conditions that $\Phi_p \mid m_{\La}$ or $\Phi_{pr} \mid m_{\La}$,   and $\Phi_{p^2} \mid m_{\La}$ or $\Phi_{p^2r} \mid m_{\La}$
        we obtain that the projection of $\La$ to $\Z_{p^2}$ is of the following form: 
        \begin{equation}\label{eqp2}
             \La_{p^2}=c \Z_{p^2}+rD, 
             \end{equation}
        where $c$ is a nonnegative integer and $D$ is a multiset on $\Z_{p^2}$. 
        If $c = 0$, then $r \mid |\La|$ and if $D=\emptyset$, then $p^2 \mid |\La|$. 
        Both cases contradict our assumption that $pq \mid \mid |S|=|\La|$. 
        Thus there are at least $r+1$ elements of $\La$ projecting to the same element of $\Z_{p^2}$ so we obtain $\Phi_q \mid m_S$. If $q <r$, then there are two among them, whose difference is divisible by $q$ as well, implying $\Phi_r \mid m_S$, a contradiction. Thus we may assume $q>r$.  
        
        Assume $\Phi_{p^2r} \mid m_{\La}$. Then $\La_{p^2r}$ is a multiset, which is the sum of $\Z_p$-cosets and $\Z_r$-cosets by Lemma \ref{lem0603}. Since $c>0$ and $D\ne0$ in equation \eqref{eqp2}, there is a $\Z_{pr}$-coset of $\Z_{p^2r}$, whose intersection with the multiset $\La_{p^2r}$ is the sum of $k$ $\Z_p$-cosets and $l$ $\Z_r$-cosets with $k+l \ge 2$. 
        
        Now we argue that $\La_{p^2r}$ contains a $\Z_r$-coset. Assume this is not the case, thus we can write  $\La_{p^2r}$ as the sum of $\Z_p$-cosets. Then the number of elements of $\La_{p^2r}$ contained in each  $\Z_{pr}$-coset is divisible by $p$.
        If $\Phi_p \mid m_{\La}$, then these numbers are the same so we would have $p^2 \mid |\La|$, a contradiction. If $\Phi_{pr} \mid m_{\La}$, then $\La_{pr}$ is the sum of $\Z_p$ and $\Z_r$-cosets. 
        If $\La_{pr}$ contains a $\Z_r$-coset, then $\La_{p^2r}$-contains a $\Z_{pr}$-coset since it is the sum of $\Z_p$-cosets so $\La_{p^2r}$ contains a $\Z_r$-coset as required. 
        If $\La_{pr}$ is the sum of $\Z_p$-cosets only, then we again have $p^2 \mid |\La|$, a contradiction.
        
        Since $c>0$ in equation \eqref{eqp2} and since we have a $\Z_{p^2r}$-coset containing a $\Z_r$-coset, which is projected on a $\Z_r$-coset contained in $\La_{p^2r}$, we have two elements  $\la_1,\la_2 \in \La_{p^2r}$ with $pr=gcd(p^2r,\la_1-\la_2) $, which we may also write as $pr \mid\mid \la_1-\la_2$. 
        It is not hard to see from the description of $\La $ in this case that for every $d \mid p^2r$ we have $\la,\la' \in \La_{p^2r}$ with $d \mid \mid \la-\la'$. 
        Then we have $\Phi_p \Phi_{r}\Phi_{pr}\Phi_{p^2}\Phi_{p^2r} \mid m_{S}$ in $\Z_q[x]$ so by projecting $S$ to $\Z_{p^2r}$ we obtain a multiset of the form $c'\Z_{p^2r}+qD'$ ($c',D' \ge0$). If $c'=0$, then $S$ is a spectral set, which is the union of $\Z_q$-cosets, hence a tile. If $c'>0$, then $D'=0$ since a $\Z_q$-coset cannot contain more than $q$ points of $\La$.
        Then $p^2r \mid |S|$, a contradiction.
        
        Thus we may assume $\Phi_{p^2r} \nmid m_{\La}$ so we have $\Phi_{p^2} \mid m_{\La}$. Then since $p^2 \nmid |S|$ we must have $\Phi_p \nmid m_{\La}$ so we have $\Phi_{pr} \mid m_{\La}$. We remind that we have already seen that $\Phi_q \mid m_S$ so $S$ is equidistributed on the $\Z_{p^2r}$-cosets. 
        Now we apply this to obtain information about the structure of $S$.
        
        We investigate the intersection of $S$ with each $\Z_{p ^2r}$-cosets. Assume $s_1,s_2 \in S$ are contained in a $\Z_{p ^2r}$-coset but are not contained in a $\Z_{pr}$-coset.
        If their difference is not divisible by $r$, then we would have $\Phi_{p^2r} \mid m_{\La}$, which we have excluded above. 
        Similarly, if $s_3 \ne s_4 \in S$ are contained in a $\Z_{pr}$-coset, then they need to have different $r$-coordinates, otherwise we would have $\Phi_p \mid m_{\La}$.
        Their difference is not divisible by $p^2$ since we would have $\Phi_r \mid m_{\La}$, which is impossible since $r \nmid |\La|=|S|$.
        Each $\Z_{p^2r}$-coset contains the same amount of elements of $S$ by $\Phi_q \mid m_S$, which is then at least $p$. The previous argument shows that each $\Z_{p^2r}$-coset contains exactly $p$  elements of $S$ and either they lie in different $\Z_{pr}$-cosets or they are all contained in one $\Z_{pr}$-coset with $p^2$ not dividing their differences.
        If for each  $\Z_{p^2r}$-coset only one of the two types appears, then $S$ is a tile. 
        
        Now we argue that $\Phi_{p^2q} \mid m_S$ or $\Phi_{n} \mid m_S$. By Proposition \ref{propreduction} we may assume $0 \in \La$ and that $\La$ is not contained in a proper subgroup of $\Z_{p^2qr}$. Then our claim follows from Lemma \ref{lempqnemoszt}.
        
         Assume  $\Phi_{p^2q} \mid m_S$. Since $r \nmid |S|=|\La|$ we have $S_{p^2q}$ is a set. Then $S_{p^2q}$ is the union of $\Z_p$ and $\Z_q$-cosets. 
        Since there is at least one $\Z_{p^2r}$-coset such that each of its $\Z_{pr}$-cosets contains one element of $S$ we have that $S_{p^2q}$ is the union of $\Z_q$-cosets all contained in different $\Z_{pq}$-cosets.
        This contradicts the existence of $\Z_{p^2r}$-cosets of $\Z_{n}$, which contains a $\Z_{pr}$-coset containing exactly $p$ elements of $S$ (any pair of these elements have different $r$-coordinate).
        
        It follows that we may assume $\Phi_{n} \mid m_S$. It is clear from our previous discussion that there are $x,y$ with $p \mid x-y$ and $q \nmid x-y$ such that $\Z_{pr}+x$ and $\Z_{pr}+y$ 
        contains $1$ and $p$ elements of $S$, respectively. 
        We remind that if it contains $p$, then in that $\Z_{pr}$-coset, the  difference of the elements of $S$ lying in this $\Z_{pr}$-coset is not divisible by either $p^2$ or with $r$.
        Then one can build up a $3$-dimensional cube in the $\Z_{pqr}$-coset containing $x$ and $y$, which contains exactly one element of $S$ or exactly two elements of $S$ of Hamming distance $2$, which contradicts the fact that $S$ satisfies the $3$-dimensional cube-rule in each $\Z_{pqr}$-cosets. 
        
        A similar argument works if $pr \mid \mid |S|$ since the role of $q$ and $r$ is symmetric.

        
        {\bf Case 3.}
        Let us assume that $qr \mid\mid |S|$. \\
        Then either $S$ is a complete set of residues $\pmod{qr}$, whence $S$ is a tile or there are two different elements of $S$, whose difference is divisible by $qr$. This would imply $\Phi_p  \mid m_{\La}$ or $\Phi_{p^2}  \mid m_{\La}$. In both of these cases we have $p  \mid |\La|=|S|$, a contradiction. 
        
        \subsection{$|S|$ has at most one prime divisors among $p,q,r$}
        Let us assume $1 \mid\mid |S|$ or $p \mid\mid |S|$ or $q \mid\mid |S|$ or $r \mid\mid |S|$. 
        
        If $\Phi_n \mid m_S$, then the intersection of $S$ with each $\Z_{pqr}$-coset satisfies the $3$-dimensional cube-rule. 
        Then by Lemma \ref{lemcuberulecorr}, we cannot have $1 \mid\mid |S|$. If $S$ is a union of $\Z_p$-cosets, $\Z_q$-cosets or $\Z_r$-cosets, respectively, then by applying Proposition \ref{propreduction} we conclude that $S$ is a tile. 
        
        A similar argument works for $\La$. If $\Phi_n \mid m_{\La}$, then $\La$ is the union of $\Z_p$-coset or $\Z_q$-cosets or $\Z_r$-cosets. Then since $(\La,S)$ is also a spectral pair we 
        have by Proposition \ref{propreduction} that $\La$ is a tile. Then $|\La| = |S|\mid n$ so we have that $|S|=|\La|=p$ or $|S|=|\La|=q$ or $|S|=|\La|=r$. By Lemma  \ref{lemcuberulecorr} $\La$ is just a $\Z_{|\La|}$-coset, which implies that $\Phi_{|\La|} \mid m_S$. Then it is easy to see that (T1) and (T2) are satisfied for $S$ so it is a tile.
        
        Thus we may assume $\Phi_n \nmid m_S$ and    $\Phi_n \nmid m_{\La}$.
        By Lemma \ref{lempqnemoszt} we have $\Phi_{p^2q} \Phi_{p^2r}\mid m_{S}$.
        
        Without loss of generality we may assume $r \nmid |S|$ since the role of $q$ and $r$ are symmetric. 
        Then $S_{p^2q}$ is a set so it is the disjoint union of $\Z_p$-cosets and $\Z_q$-cosets. 
        By Proposition \ref{propreduction} we have that $\langle S \rangle= \Z_n$ so $\langle S_{p^2q} \rangle= \Z_{p^2q}$. It follows using $0 \in S_{p^2q}$ that that at least two $\Z_{pq}$-cosets of $\Z_{p^2q}$ contain elements of $S_{p^2q}$.
        Note that the same argument works for $\La_{p^2q}$ as well.
        We remind that $S_{p^2q}$
        is the union of $\Z_p$-cosets and $\Z_q$-cosets. 
        
        Assume there  is a $\Z_{pq}$-coset in $\Z_{p^2q}$, whose intersection with $S_{p^2q}$
        contains a $\Z_q$-coset and another $\Z_{pq}$-coset containing a $\Z_p$-coset also contained in $S_{p^2q}$. 
        Then for every $d \mid p^2q$ with $d \ne p^2q$, $d \ne p$ there are $s_{d,1},s_{d,2} \in S$ such that $d \mid \mid \pi_{p^2q}(s_{d,1})-\pi_{p^2q}(s_{d,2})$, where $\pi_{p^2q}$ is the natural projection of $\Z_n$ to $\Z_{p^2q}$. Hence $\Phi_{d} \mid m_{\La}$ or $\Phi_{dr} \mid m_{\La}$ for every $ d \mid p^2q$ and $d \ne 1$, $d \ne pq$. Note that the exception for $d=p$ only applies if the intersection of each $\Z_{pq}$-coset in $\Z_{p^2q}$ with $S_{p^2q}$ is either empty or a $\Z_p$-coset or a $\Z_q$-coset 
        
        It follows from Lemma \ref{lempqnemoszt} that $\Phi_{qr} \mid m_{\Lambda}$ or $\Phi_{pqr} \mid m_{\Lambda}$ since $\Phi_{n} \mid m_{\Lambda}$ has already been excluded.

        Let us assume $\Phi_{pqr} \mid m_{\Lambda}$.
        Then $\Phi_{d} \mid m_{\La}$ or $\Phi_{dr} \mid m_{\La}$ for every $ d \mid p^2q$ and $d \ne 1$. Thus by Lemma \ref{lemprojekcio} we have that $\La_{p^2q}=c\Z_{p^2q}+rD$. If $c>0$, then $|S|\ge p^2q$. In this case $S$ is a tile or $r \mid |S|$ by the argument used in Case 1.  
        If $c=0$, then we obtain $r \mid |S|$, a contradiction. 
        
        It follows we may assume $\Phi_{qr} \mid m_{\Lambda}$ and $\Phi_{pqr} \nmid m_{\Lambda}$. Then $\Lambda_{qr}$ is the sum of $\Z_q$-cosets and $\Z_r$-cosets. If both types appear, then $\Phi_p \mid m_S$ or $\Phi_{p^2} \mid m_S$. The fact that $\Phi_p \mid m_S$ implies that   the intersection of each $\Z_{pq}$-coset in $\Z_{p^2q}$ with $S_{p^2q}$ is of the same cardinality, which does not hold in our case since $p \ne q$. On the other hand, $\Phi_{p^2} \mid m_S$ would imply that $\Lambda_{p^2}$ is equidistributed in each $\Z_p$-coset, which does not hold since there is a $\Z_{pq}$-coset in $\Z_{p^2q}$ that intersects $S_{p^2q}$ in a single $\Z_q$-coset. 
        Now we have that $\Lambda_{qr}$ is the disjoint union of $k\ge 2$ $\Z_q$-cosets  by Proposition \ref{propreduction} \ref{itemreductiona}. If $\lambda$ and $\lambda'$ are elements of $\Lambda_{qr}$ of Hamming distance $2$, then for the elements $\overline{\lambda}$ and $\overline{\lambda'}$ of $\Lambda$ projecting to these elements we have $qr \mid \mid \overline{\lambda} - \overline{\lambda'}$ since $\Phi_{pqr} \nmid m_S$ and $\Phi_{n} \nmid m_S$.
    We obtain that $\overline{\lambda}$ and $\overline{\lambda'}$ are contained in the same
    $\Z_{pqr}$-coset. This holds for every element of $\Lambda$ if $q >2$ or if $k >2$,
    when Proposition \ref{propreduction} \ref{itemreductiona} gives that $S$ is a tile.
    The remaining case is $|S|=4 $, which is handled by Theorem 2.1 in
    \cite{KMkishalmaz}.
        
        Assume now that the intersection of $S_{p^2q}$ with each $\Z_{pq}$-coset is the union of (possible 0) $\Z_q$-cosets and this intersection is nonempty for at least two cosets of $\Z_{pq}$. 
        Then it follows from $\Phi_n \nmid m_{\La}$ that if $x,y \in S$, whose natural projections to $\Z_{p^2q}$ are not contained in a proper coset of $\Z_{pq}$, then their difference is divisible by $r$. If $q>2$, then it follows from $\langle S_{p^2q}\rangle=\Z_{p^2q}$ that the difference of any two elements of $S$ is divisible by $r$ so it is contained in a proper coset of $\Z_n$. Proposition \ref{propreduction} gives that $S$ is a tile in this case. 
        
        If $q=2$ but there are more than two $\Z_{pq}$-cosets containing elements of $S_{p^2q}$, then we build up a graph $\Gamma$ having vertex set $S$ and two vertices are adjacent if and only if their difference is not divisible by either $p$ or $q$. Again, the difference of two adjacent vertices is divisible by $r$. It is not hard to verify that $\Gamma $ is connected 
        and then $S$ is contained in a $\Z_{p^2q}$-coset of $\Z_n$ so it is a tile. 
        
        If there is a 
        $\Z_{pq}$-coset in $\Z_{p^2q}$, whose intersection with $S_{p^2q}$ contains at least two $\Z_{q}$-cosets and another one which contains at least one $\Z_{q}$-coset, then again we have that for every $1 \ne d \mid p^2q$ there are elements $s_1,s_2$ of $S_{p^2q}$ such that $d \mid \mid s_1-s_2$, which case has already been handled above. 
        
        The $|S|=4$ case follows from a Theorem 2.1 of Kolountzakis and Matolcsi \cite{KMkishalmaz}, which says that spectral sets of cardinality at most $5$  in finite abelian groups are tiles.

        Thus it remains that $S_{p^2q} $ is the union of $\Z_p$-cosets only. We also have that these $\Z_p$-cosets are not contained in a $\Z_{pq}$-coset or a $\Z_{p^2}$-coset of $\Z_{p^2q}$. 
        For every $s \in S_{p^2q}$ the unique element of $S$ projecting to $s$ and it is denoted by $\bar{s}$.
        Assume that for every $x \in S_{p^2q}$ there is $y \in S_{p^2q}$ such that $p \nmid x-y$ and $q \nmid x-y$. Then for every $x' \in x+\Z_p \subset S_{p^2q}$, we have $p \nmid x'-y$ and $q \nmid x'-y$. Since $\Phi_n \nmid m_{\La}$ we have that $r \mid \bar{x}-\bar{y}$ and $\bar{x'}-\bar{y}$ so $r \mid \bar{x}-\bar{x'}$.
        The same holds for every element of $x+\Z_p$ so $\{ \overline{x+u}  \colon~ u \in \Z_p \}$ is a $\Z_p$-coset. Therefore $S$ is the union of $\Z_p$-cosets, which is handled by Proposition \ref{propreduction}.
        
        If there is a $x \in S_{p^2q}$ such that $p \mid x-y$ or $q \mid x-y$ for every $y \in S_{p^2q}$, then $S_{p^2q}$ is contained in $(x+\Z_{pq}) \cup (x+\Z_{p^2})$. Since $S_{p^2q}$ is not contained in any of this two sets appearing in the union we again have that for every $d \mid p^2q$ with $d \ne p^2q$ there are $x,y \in S_{p^2q}$ with $d \mid \mid x-y$, which case has already been settled.   
        
        \section*{Acknowledgement}
        The author is grateful to Gergely Kiss, Romanos Diogenes Malikiosis and M\'at\'e Vizer. Many of the ideas appearing in the manuscript come up during fruitful discussions aiming to prove other results concerning Fuglede's conjecture or in our joint paper \cite{negyen}. In particular, I learned a version of Lemma \ref{lemprojekcio} from Romanos Diogenes Malikiosis and an earlier version of Lemma \ref{lemcuberulecorr} was worked out by Gergely Kiss and M\'at\'e Vizer.
\bibliographystyle{amsplain}

%
%

\begin{dajauthors}
\begin{authorinfo}[pgom]
  G\'abor Somlai\\
  E\"otv\"os Lor\'and University\\
  Budapest, Hungary\\
  gabor.somlai\imageat{}ttk\imagedot{}elte\imagedot{}hu \\
  \url{http://zsomlei.web.elte.hu}
\end{authorinfo}
  \end{dajauthors}
\end{document}